\documentclass{mfatshort}

\makeatother
\newtheorem{lemma}{Lemma}
\newtheorem{theorem}{Theorem}[section]

\newtheorem{corollary}{Corollary}
\newtheorem{definition}{Definition}

\newtheorem{Example}{Example}
\begin{document}

\title[Controlled Fusion Frames]
	{Some properties of Controlled Fusion Frames}

\author{H. Shakoory}
\address{Department of Mathematics\\ Shabestar Branch\\ Islamic Azad University\\ Shabestar\\ Iran.}
\curraddr{}
\email{Habibshakoory@yahoo.com}
\thanks{}

\author{R. Ahmadi}
\address{Institute of Fundamental Sciences\\University of Tabriz\\, Iran\\}
\email{rahmadi@tabrizu.ac.ir}
\thanks{}

\author{N. Behzadi}
\address{Institute of Fundamental Sciences\\University of Tabriz\\, Iran\\}
\email{n.behzadi@tabrizu.ac.ir}
\author{S. Nami}
\address{Faculty of Physic\\University of Tabriz\\, Iran\\}
\email{S.Nami@tabrizu.ac.ir}
\date{}
\dedicatory{}

\begin{abstract}
Controlled frames in Hilbert spaces  have been  introduced by Balazs, Antoine and Grybos  to improve the numerical output of in relation to algorithms for inverting the frame operator.  In this paper we have introduced and displayed some new concepts and results on controlled  fusion frames for Hilbert spaces. It is shown that controlled fusion frames as a generalization of fusion frames give a generalized way to obtain numerical advantage in the sense of reconditioning to check the fusion frame condition. For this end, we introduce the notion of Q-duality for Controlled fusion frames. Also, we survey the robustness of Controlled fusion frames under some perturbations.

\end{abstract}
\maketitle
\section{Introduction}
Frames, as a expansion of the bases in Hilbert spaces, were first introduced
by Duffin and Schaeffer During their study of nonharmonic Fourier series in 1952, they (\cite{Duffin}) introduced frames as a expansion of the bases in Hilbert spaces. 
recently, frames play an fundamental role not only in the visionary  but also in many kinds of applications and have been widely applied in filter bank theory, coding and communications, signal processing, system modeling, see(\cite{Musazadeh}),(\cite{Strohmer}),({\cite{Feichtinger}}), ({\cite{cassaza}}), ({\cite{Blo}}).

One of the newest generalization of frames is controlled frames. Controlled frames have been introduced  to improve the numerical efficiency of interactive algorithms for inverting the frame operator on abstract Hilbert spaces (\cite{Balazs}), (\cite{Khosravi}),(\cite{Hua}).

This maniscript is organized as follows. In section 2, we remined some definitions and Lemmas in frames and operators theory. In section 3, we fix the notations of this paper, summarize known and prove some new results. In section 4, we defined Q-duality and perturbation for $CC^{\prime}$ -Controlled fusion frame and express some results about them.

Throughout this paper,  $H$ is a separable Hilbert space, $\mathcal{B}(H)$ is the family of all bounded linear operators on $H$ and $GL(H)$ denotes the set of all bounded linear operators which have bounded inverse. Let $GL^{+}(H)$ be the set of all positive operators in $GL(H)$.

It is easy to check that if $C, C^{\prime} \in GL(H)$, then  $C^{\prime*}, C^{\prime-1}$ and $CC^{\prime}$ are  in $GL(H)$.
Assume that $Id_{H}$ be the identity operator on $H$  and $\pi_W$ be the orthogonal projection from $H$ onto a closed subspace $V\subseteq H$.
\section{Preliminaries}
In this section, some necessary definitions and lemmas are introduced.\\
\begin{definition}
A sequence $\lbrace f_{i} \rbrace _{i\in I}$ in  $H$ is a frame  if there exist constants $0 <A \leqslant B <\infty$ such that for all $f\in H$
\begin{align*}
A\Vert f \Vert ^{2}\leqslant \sum _{i\in I} \vert\langle f,f_{i}\rangle\vert^{2}\leqslant B \Vert f \Vert ^{2}.
\end{align*}
\end{definition}
The constants $A,B$ are called frame bounds; $A$ is the lower bound and $B$ is the upper bound. The frame is thight if $A=B$, it is called a Parseval frame if $A=B=1$. If  we only have the upper bound, We call $\lbrace f_{i} \rbrace _{i\in I}$ a Bessel sequence. If $\lbrace f_{i} \rbrace _{i\in I}$ is a Bessel sequence then the following operators are bounded:
\begin{align*}
T&:\ell^{2}(I)\longmapsto H\\
&T(c_{i})=\sum _{i\in I}c_{i}f_{i}
\end{align*}
\begin{align*}
T^{*}&:H\longmapsto\ell^{2}(I)\\
&T^{*}f=\lbrace\langle f,f_{i}\rangle\rbrace _{i\in I}
\end{align*}
\begin{align*}
S&:H\longmapsto H\\
Sf&=TT^{*}f=\sum _{i\in I} \langle f,f_{i}\rangle f_{i}.
\end{align*}
These operators are called synthesis operator; analysis operator and frame operator,   respectively. The representation space employed in this setting is:
\begin{align*}
\ell^{2}(I)=\Big\lbrace\lbrace c_{i}\rbrace_{i\in I}:\ c_{i}\in\Bbb C ,\ \sum _{i\in I}\vert c_{i} \vert ^{2}<\infty \Big\rbrace
\end{align*}
\begin{definition}.
Let $W:=\{W_i\}_{i\in I}$ be a family of closed subspaces of $H$ and $v:=\{v_i\}_{i\in I}$ be a family of weights (i.e. $v_i>0$ for any $i\in I$). We say that $W$ is a fusion frame with respect to $v$ for $H$ if there exists $0<A\leq B<\infty$ such that for each $h\in H$
\begin{eqnarray*}
A\Vert h\Vert^2\leq\sum_{i\in I}v^{2}_i\Vert \pi_{W_i}(h)\Vert^2\leq B\Vert h\Vert^2.
\end{eqnarray*}
\end{definition}
\begin{lemma}(\cite{ga})\label{l1}
Let $V\subseteq H$ be a closed subspace, and $T$ be a linear  bounded operator on $H$. Then
$$\pi_{V}T^*=\pi_{V}T^* \pi_{\overline{TV}}.$$
If $T$ is a unitary (i.e. $T^*T=Id_{H}$), then
$$\pi_{\overline{Tv}}T=T\pi_{v}.$$
\end{lemma}
\begin{lemma}(\cite{ch})\label{l3}
Let $u \in \mathcal{B}(K,H)$ be a bounded operator with closed range $R_{u}$. Then there exist a bounded  operator $u^{\dagger} \in \mathcal{B}(H,K)$ for which\\
\begin{align*}
uu^{\dagger}x=x, \ \  x \in R_{u}
\end{align*}
and $$(u^*)^\dagger=(u^{\dagger})^*.$$
\end{lemma}
\section{Controlled fusion frame}
\begin{definition}.\cite{Khosravi}
Let $\lbrace W_{i}\rbrace_ {i\in I}$ be a collection of closed subspace in Hilbert space $H$,  $\lbrace v_{i}\rbrace_ {i\in I}$ be a family of weights, i.e. $v_{i}>0$,  $i \in I$ and $C, C'\in GL(H)$. The sequence  $W=\lbrace (W_{i},v_{i})\rbrace _ {i \in I}$ is called a fusion frame controlled by $C$ and $C^{\prime}$ or $CC^{\prime}$-Controlled fusion frame  for $H$ if there exist constants $0< A \leq B< \infty$ such that for all $f \in H$
\begin{align*}
A \Vert f \Vert^{2} \leq \sum _{i\in I} v_{i}^{2} \langle \pi_{W_{i}} C^{\prime}f,\pi_{W_{i}} Cf \rangle  \leq B \Vert f\Vert^{2}
\end{align*}
\end{definition}
Throughout this paper, $W$ will be a set $\lbrace (W_{i},v_{i})\rbrace _ {i \in I}$ unless otherwise stated.  $W$ is called a tight controlled fusion frame, if  the constants $A,B$ can be chosen such that  $A=B$, a parseval fusion frame provided $A=B=1$. We call $W$ is a $C^2$-Controlled fusion frame  if $C=C^{\prime}$.\\ 
If only the second Inequality is required, We call $W$-Controlled Bessel fusion sequence  with bound $B$.\\
We define the controlled analysis operator by\\
\begin{align*}
&T_{W}:H \rightarrow \mathcal{K}_{2,W} \\
&T_{W}(f)=(v_{i}(C^{*}\pi_{W_{i}}C^{\prime})^{\frac{1}{2}}f).
\end{align*}
where
\begin{align*}
\mathcal{K}_{2,W}:=\lbrace v_i(C^{*}\pi_{W_{i}} C^{\prime})^{\frac{1}{2}}f \ : \ f\in H\rbrace \subset (\bigoplus_{i \in I} H)_{l^{2}}.
\end{align*}
It is easy to see that $\mathcal{K}_{2,W}$ is closed and $T_{W}$ is well defined. Moreover $T_{W}$ is a bounded linear operator with the adjoint operator $T ^{*}_{W}$ defined by
\begin{align*} 
&T^{*}_{W}:\mathcal{K}_{2,W} \rightarrow H \\
&T ^{*}_{W}(v_i(C^{*}\pi_{W_{i}} C^{\prime})^{\frac{1}{2}}f)=\sum _{i\in I} v_{i}^{2}C^{*}\pi_{W_{i}} C^{\prime}f.
\end{align*}
Therefore, we define the controlled fusion frame operator $S_{W}$ on $H$ by
\begin{align*}
S_{W}f=T ^{*}_{W}T_{W}(f)=\sum _{i\in I} v_{i}^{2}C^{*}\pi_{W_{i}} C^{\prime}f.
\end{align*}
\begin{Example}
Let $\lbrace e_{1},e_{2},e_{3}\rbrace$ be the standard orthonormal basis for $\mathcal{R}^{3}$ and $W_{1}=\overline{span}\lbrace e_{1},e_{2} \rbrace$, $W_{2}=\overline{span}\lbrace e_{1},e_{3} \rbrace$, $W_{3}=\overline{span}\lbrace e_{2},e_{3} \rbrace$. Let
$C(x_{1},x_{2},x_{3})=(x_{1},x_{2},x_{1}+x_{3})$, $C^{\prime}(x_{1},x_{2},x_{3})=(x_{1},x_{2},x_{2}+x_{3})$ be two operators on $\mathcal{R}^{3}$. It is easy to see that $C,C^{\prime}\in GL(\mathcal{R}^{3})$.  Now an easy computation shows that $\lbrace(W_{i},1)\rbrace^{3} _{i=1}$,  is a $CC'$-controlled fusion frame with bounds 1, 4. 
\begin{align*}
1 \Vert f \Vert^{2} \leq \sum _{i\in I} v_{i}^{2} \langle \pi_{W_{i}} C^{\prime}f,\pi_{W_{i}} Cf \rangle  \leq 4 \Vert f\Vert^{2}
\end{align*}
\end{Example}
\begin{theorem}\label{1}
$W$ be a  $CC'$-controlled fusion Bessel sequence  for $H$  with bound $B$ if and only if the operator
\begin{align*}
&T^{*}_{W}:\mathcal{K}_{2,W} \rightarrow H \\
&T ^{*}_{W}(v_i(C^{*}\pi_{W_{i}} C^{\prime})^{\frac{1}{2}}f)=\sum _{i\in I} v_{i}^{2}C^{*}\pi_{W_{i}} C^{\prime}f.
\end{align*}
is well -defined and bounded operator with $\Vert T^{*}_{W} \Vert \leq \sqrt{B}$.\\
\end{theorem}
\begin{proof}
The necessary condition follows from the definition of $CC'$-controlled fusion Bessel sequence. We only need to prove that the sufficient condition holds. Let $T^{*} _{W}$ be well-defined and bounded operator with $\Vert T^{*}_{W} \Vert \leq  \sqrt{B}$. For any $f\in H$, we have
\begin{align*}
(\sum_{i \in I} v_{i} ^{2} \langle  \pi_{W_{i}}C^{\prime}f,  \pi_{W_{i}} Cf \rangle )^{2}&=(\sum_{i\in I} v_{i} ^{2} \langle C^{*}\pi_{W_{i}}  C^{\prime}f,f \rangle )^{2}\\
&= (\langle T^{*}_{W}\big(v_i(C^{*}  \pi_{W_{i}} C')^{\frac{1}{2}}f\big), f \rangle )^2\\
&\leq \Vert T^{*}_{W}\Vert ^{2} \Vert (v_i(C^{*}\pi_{W_{i}}   C')^{\frac{1}{2}} f \Vert ^{2} \Vert f \Vert ^{2}.
\end{align*}
But
\begin{align*}
\Vert (v_i(C^{*}  \pi_{W_{i}} C')^{\frac{1}{2}} f\Vert_{2} ^{2}=\sum _{i \in I} v_{i}^{2}\langle \pi_{W_{i}}  C'f,\pi _{W_{i}} C f\rangle.
\end{align*}
It follows that
\begin{align*}
\sum _{i \in I} v_{i}^{2}\langle \pi_{W_{i}}  C'f,\pi _{W_{i}} C f\rangle\leq B\Vert f\Vert ^{2}.
\end{align*}
and this means that $W$ is a  $CC'$-controlled fusion Bessel sequence for H.
\end{proof}
\begin{theorem}\label{2}
$W$ is a $CC^{\prime}$-Controlled fusion frame  for $H$ if and only if 
\begin{align*}
&T^{*}_{W}:\mathcal{K}_{2,W} \rightarrow H \\
&T ^{*}_{W}(v_i(C^{*}\pi_{W_{i}} C^{\prime})^{\frac{1}{2}}f)=\sum _{i\in I} v_{i}^{2}C^{*}\pi_{W_{i}} C^{\prime}f.
\end{align*}
is a well-defined, bounded and surjective.
\end{theorem}
\begin{proof}
If $W$ is a $CC^{\prime}$-Controlled fusion frame  for $H$, the operator $S_{W}$ is invertible. Thus, $T^{*}_{W}$ is surjective. 

Conversely, let $T^{*}_{W}$ be a well-defined, bounded and surjective. Then, by Theorem \ref{1}, $W$ is a $CC^{\prime}$-Controlled Bessel fusion sequence for $H$.\\
So, $T_{W}(f)=(v_{i}(C^{*}\pi_{W_{i}}C^{\prime})^{\frac{1}{2}}f)$ for all $f \in H$. Since $T^{*}_{W}$  is Surjective, by Lemma [\ref{l3}], there exists an operator $T^{*\dagger}_{W}:H \rightarrow \mathcal{K}_{2,W}$ such that $T^{\dagger}_{W} T_{W}=I_{H}$ . Now, for each $f\in H$ we have 
\begin{align*}
\Vert f\Vert^{2}\leq \Vert T^{\dagger}_{W}\Vert ^{2}.\Vert T_{w}f\Vert ^{2}=\Vert T^{\dagger}_{W}\Vert ^{2}.\sum _{i\in I} v_{i}^{2} \langle \pi_{W_{i}} C^{\prime}f,\pi_{W_{i}} Cf \rangle ^{2}
\end{align*}
Therefore, $W$ is a $CC^{\prime}$-Controlled fusion frame  for $H$ with the lower Controlled fusion frame bound  $\Vert T^{\dagger}_{W}\Vert ^{-2}$ and the upper Controlled fusion frame $\Vert T^{*}_{W}\Vert ^{2}$.
\end{proof}
\begin{theorem}
Let $W$ be a $C^{2}$-controlled fusion frame with frame bounds $A$ and $B$. If $u \in \mathcal{B}(H)$ is an invertible  operator such that $u^*C=Cu^*$, then
$\lbrace(uW_{i},v_{i})\rbrace _{i \in I}$ is a $C^{2}$-controlled
fusion frame for $H$.
\end{theorem}
\begin{proof}
Let $f\in H$. From lemma \ref{l1}, we have
\begin{align*}
\Vert\pi_{W_i}Cu^*f\Vert=\Vert\pi_{W_i}u^*Cf\Vert=\Vert\pi_{W_i}u^*\pi_{uW_i}C^*f\Vert\leq
\Vert u\Vert \Vert \pi_{uW_i}C^*f\Vert.
\end{align*}
Therefore,
\begin{align*}
A\Vert u^*f\Vert^2\leq\sum_{i\in I}\Vert\pi_{W_i}Cu^*f\Vert^2\leq\Vert u\Vert^2\sum_{i\in I} \Vert \pi_{uW_i}C^*f\Vert.
\end{align*}
But,
\begin{align*}
\Vert f\Vert^2\leq\Vert (u^{-1})^{*}u^*f\Vert^2\leq\Vert u^{-1}\Vert^2\Vert u^*f\Vert^2.
\end{align*}
Then,
\begin{align*}
A\Vert u^{-1}\Vert^{-2}\Vert u\Vert^{-2}\Vert f\Vert^2\leq\sum_{i\in I} \Vert \pi_{uW_i}C^*f\Vert.
\end{align*}
On the other hand, from lemma \ref{l1}, we obtain, with $u^{-1}$ instead of $T$:
$$\pi_{uW_i}=\pi_{uW_i}(u^*)^{-1}\pi_{W_i}u^*.$$
Thus,
$$\Vert\pi_{uW_i}Cf\Vert\leq\Vert u^{-1}\Vert \Vert\pi_{W_j}u^*Cf\Vert,$$
and it follows 
\begin{align*}
\sum _{i \in I} v_{i}^{2} \Vert  \pi _{uW_{i}}Cf\Vert ^{2}&\leq \Vert u^{-1}\Vert^{2}\sum _{i \in I} v_{i}^{2}\Vert
\pi _{W_{i}}u^{*}Cf\Vert ^{2}\\
&=\Vert u^{-1}\Vert^2\sum_{i\in I}v_i^2\Vert \pi_{W_i}C u^{*}f\Vert^2\\
&\leq B \Vert  u^{-1}\Vert^{2}\Vert  u \Vert^{2} \Vert f\Vert^{2}.
\end{align*}
\end{proof}
\begin{theorem}
Let $W=\lbrace(W_{i},v_{i})\rbrace _{i \in I}$ be a $C^{2}$-controlled
fusion frame with frame bounds $A$ and $B$. If $u \in \mathcal{B}(H)$is an
invertible and unitary operator such that $uC=Cu$, then
$\lbrace(uW_{i},v_{i})\rbrace _{i \in I}$ is a $C^{2}$-controlled
fusion frame for $H$.
\end{theorem}
\begin{proof}
Using lemma \ref{l1}, we have foe any $f\in H$,
\begin{align*}
A\Vert f\Vert ^{2}&\leq\Vert u\Vert^2 \Vert u^{-1}f\Vert^2\\
&\leq\Vert u\Vert^2\sum _{i \in I} v_{i}^{2}\Vert \pi _{W_{i}}
u^{-1}Cf\Vert ^{2}\\
&\leq\Vert u\Vert^{2} \sum _{i \in I} v_{i}^{2}\Vert u^{-1}\pi
_{uW_{i}}Cf\Vert ^{2}\\
&\leq\Vert u\Vert^2 \Vert u^{-1}\Vert^2\sum_{i\in I}v_i^2\Vert
\pi_{uW_i}Cf\Vert^2,
\end{align*}
and we obtain
\begin{align*}
 \sum _{i \in I} v_{i}^{2}\Vert \pi _{uW_{i}}Cf\Vert ^{2}\geq\dfrac{A}{\Vert u \Vert^{2}\Vert u^{-1}\Vert^{2}}\Vert
 f\Vert^{2}.
\end{align*}
On the other hand, from lemma \ref{l1}, we obtain
\begin{align*}
\sum _{i \in I} v_{i}^{2} \Vert  \pi _{uW_{i}}Cf\Vert ^{2}&\leq \Vert u\Vert^{2}\sum _{i \in I} v_{i}^{2}\Vert
\pi _{W_{i}}u^{-1}Cf\Vert ^{2}\\
&=\Vert u\Vert^2\sum_{i\in I}v_i^2\Vert \pi_{W_i}C u^{-1}f\Vert^2\\
&\leq B \Vert  u^{-1}\Vert^{2}\Vert  u \Vert^{2} \Vert f\Vert^{2}.
\end{align*}
\end{proof}
\begin{theorem}
Let  $W=\lbrace(W_{i},v_{i})\rbrace _{i \in I}$ and $Z=\lbrace(Z_{i},v_{i})\rbrace _{i \in I}$  be $CC'$-Controlled Bessel fusion sequences for $H$. Suppose that there exists $0<\epsilon< 1$ such that 
\begin{align*}
\Vert f-T^{*}_{Z}T_ {W}f \Vert \leq \epsilon \Vert f \Vert
\end{align*}
Then $W$ and $Z$  are $CC'$-Controlled fusion frame for $H$.
\end{theorem}
\begin{proof}
For each $f\in H$, we have 
\begin{align*}
\Vert T^{*}_{Z}T_ {W}f\Vert\geq\Vert f\Vert-\Vert f-T^{*}_{Z}T_ {W} f \Vert\geq (1- \epsilon)\Vert f \Vert.
\end{align*}
Therefore
\begin{align*}
(1- \epsilon)\Vert f \Vert\leq\Vert T^{*}_{Z}T_ {W}f\Vert &=\sup _ {\Vert g\Vert =1}\vert \langle T^{*}_{Z}T_ {W}f,g \rangle\vert\\
&=\sup _ {\Vert g\Vert =1}\vert \langle T_ {W}f,T_{Z}g \rangle\vert\\
&\leq \sup _ {\Vert g\Vert =1} \Vert T_ {W}f\Vert .\Vert T_{Z}g\Vert\\
&\leq \sqrt{B}(\sum _{i\in I} v_{i}^{2} \langle \pi_{W_{i}} C^{\prime}f,\pi_{W_{i}} Cf \rangle)^{\frac{1}{2}},
\end{align*}
where $B$ is a Controlled Bessel bound for $W$. Hence,
\begin{align*}
\dfrac{(1- \epsilon)^{2}}{B}.\Vert f\Vert^{2}\leq(\sum _{i\in I} v_{i}^{2} \langle \pi_{W_{i}} C^{\prime}f,\pi_{W_{i}} Cf \rangle).
\end{align*}
Therefore, $W$ is a $CC'$-Controlled fusion frame for $H$. Similarly, we can show that $Z$ is also a   $CC'$-Controlled fusion frame  for $H$.
\end{proof}
\begin{corollary}
Let $W:=\lbrace(W_{i},v_{i})\rbrace _{i \in I}$ and $Z:=\lbrace(Z_{i},z_{i})\rbrace _{i \in I}$ be two $CC'$-controlled fusion Bessel sequence for $H$ with bounds $B_{1}$ and $B_{2}$, respectively. Suppose that $T^{*}_{W}$ and $T^{*}_{Z}$ be their controlled analysis operators such that $T^{*}_{Z}T_{W}=Id_{H}$. Then, both $W$  and $Z$ are $CC'$-controlled fusion frame  for $H$.\\
\end{corollary}
\begin{theorem}
Let $W$ be a $CC'$-controlled fusion frame with bounds $A,B$ for $H$. Also, let $Z:=\lbrace Z_{i}\rbrace _{i \in I}$ be a family of closed subspaces in $H$ and
\begin{align*}
\Vert v_{i}(C^{*}\pi_{W_{i}} C^{\prime}-C^{*}\pi_{Z_{i}} C^{\prime})^{\frac{1}{2}}f \Vert \leq \epsilon \Vert f\Vert,
\end{align*}
for some $0<\epsilon<\sqrt{A}$. Then $Z:=\lbrace(Z_{i},v_{i})\rbrace _{i \in I}$ is a  $CC'$-controlled fusion frame with bounds $(A-\epsilon^{2})$ and $(B+\epsilon^{2})$.
\end{theorem}
\begin{proof}
For every  $f \in H$, we can write
\begin{align*}
\Vert v_{i}(C^{*}\pi_{Z_{i}} C^{\prime})^{\frac{1}{2}}f \Vert^{2}&\leq\Vert v_{i}(C^{*}\pi_{W_{i}} C^{\prime})^{\frac{1}{2}}f \Vert^{2}+\Vert v_{i}(C^{*}\pi_{W_{i}} C^{\prime}-C^{*}\pi_{Z_{i}} C^{\prime})^{\frac{1}{2}}f \Vert^{2}\\
&\leq(B+\epsilon^{2})\Vert f\Vert^{2}
\end{align*}
Thus,
\begin{align*}
\sum _{i\in I} v_{i}^{2} \langle \pi_{Z_{i}} C^{\prime}f,\pi_{Z_{i}} Cf \rangle =\Vert v_{i}(C^{*}\pi_{Z_{i}} C^{\prime})^{\frac{1}{2}}f \Vert^{2}\leq(B+\epsilon^{2})\Vert f\Vert^{2}.
\end{align*}
Therefore, $Z:=\lbrace(Z_{i},z_{i})\rbrace _{i \in I}$ is a Controlled Bessel fusion sequence. On the other hand
\begin{align*}
\Vert v_{i}(C^{*}\pi_{Z_{i}} C^{\prime})^{\frac{1}{2}}f \Vert^{2}&\geq\Vert v_{i}(C^{*}\pi_{W_{i}} C^{\prime})^{\frac{1}{2}}f \Vert^{2}-\Vert v_{i}(C^{*}\pi_{W_{i}} C^{\prime}-C^{*}\pi_{Z_{i}} C^{\prime})^{\frac{1}{2}}f \Vert^{2}\\
&\geq(A-\epsilon^{2})\Vert f\Vert^{2}.
\end{align*}
Hence,
\begin{align*}
\sum _{i\in I} v_{i}^{2} \langle \pi_{Z_{i}} C^{\prime}f,\pi_{Z_{i}} Cf \rangle =\Vert v_{i}(C^{*}\pi_{Z_{i}} C^{\prime})^{\frac{1}{2}}f \Vert^{2}\geq(A-\epsilon^{2})\Vert f\Vert^{2}
\end{align*}
and the proof is completed.
\end{proof}
\section{ Q-dual and perturbation  on Controlled fusion frame}
\begin{definition}
Assume that $W$ be a $CC^{\prime}$-Controlled fusion frame  for H. We call a  $CC'$-controlled fusion Bessel sequence  as $\tilde{W}:=\{(\tilde{W}_i, z_i)\}_{i\in I}$ the Q-dual $CC'$-Controlled fusion frame  of  $W$,  if there exists a bounded linear operator $Q:\mathcal{K}_{2,W}\longrightarrow \mathcal{K}_{2,\tilde{W}}$ such that
\begin{align*}
T^{*}_{W}QT_{\tilde{W}}=I_{H}.
\end{align*}
\end{definition}
\begin{theorem}
Let  $\tilde{W}$ be Q-dual $CC'$-Controlled fusion frame for $W$  and  $Q:\mathcal{K}_{2,W}\longrightarrow \mathcal{K}_{2,\tilde{W}}$. Then, the following conditions are equivalent.
\begin{enumerate}
\item $T^{*}_{\tilde{W}}Q^{*} T_ {W}=I_{H}$;
\item $T^{*}_ {W}QT_{\tilde{W}}=I_{H}$;
\item $\langle f,g \rangle = \langle Q^{*} T_ {W}f,T_{\tilde{W}}g \rangle=\langle Q T_{\tilde{W}}f,T_ {W}g \rangle$.
\end{enumerate}
\end{theorem}
\begin{proof}
Straightforward.
\end{proof}
\begin{theorem}
If  $\tilde{W}$ is a Q-dual for $W$, then  $\tilde{W}$ is a $CC'$-controlled fusion frame  for H.
\end{theorem}
\begin{proof}
Let $f\in H$ and $B$ an upper bound for $W$. Therefore
\begin{align*}
\Vert f\Vert ^{4}&=\vert\langle f,f \rangle\vert^{2}\\
&=\vert\langle Q^{*} T_ {W_{i}}f, T_{\tilde{W}_{i}}f\rangle\vert^{2}\\
&=\vert\langle QT_{\tilde{W}_{i}} f,T_ {W_{i}}f\rangle\vert^{2}\\
&\leq \Vert T_ {\tilde{W}_{i}}f \Vert^{2} \Vert Q\Vert ^{2} \Vert T_ {W_{i}}f \Vert ^{2}\\
&\leq \Vert T_ {\tilde{W}_{i}}f \Vert^{2} \Vert Q \Vert^{2} B \Vert f \Vert^{2}\\
&=\Vert Q \Vert^{2} B \Vert f \Vert^{2}\sum _{i\in I} z_{i}^{2} \langle \pi_{\tilde{W}_{i}} C^{\prime}f,\pi_{\tilde{W}_{i}} Cf \rangle ^{2}.
\end{align*}
Hence,
\begin{align*}
B^{-1} \Vert Q \Vert^{-2}\Vert f\Vert ^{2}\leq \sum _{i\in I} z_{i}^{2} \langle \pi_{\tilde{W}_{i}} C^{\prime}f,\pi_{\tilde{W}_{i}} Cf \rangle ^{2}
\end{align*}
and this completes the proof.
\end{proof}
\begin{corollary}
If $C_{op}$ and $D_{op}$ are the optimal bounds of $\tilde{W}$, then
$$C_{op}\geq B_{op}^{-1}\Vert Q\Vert^{-2}\ \ \ and  \ \ \ D_{op}\geq A_{op}^{-1}\Vert Q\Vert^{-2}$$
which $A_{op}$ and $B_{op}$ are the optimal bounds of $W$, respectively.
\end{corollary}
\begin{definition}
Let   $W:=\lbrace(W_{i},v_{i})\rbrace _{i \in I}$ and  $Z:=\lbrace(Z_{i},v_{i})\rbrace _{i \in I}$ be  $CC'$-controlled fusion frame for $H$  where $C,C^{\prime} \in GL(H)$ and $0\leq \lambda_{1} , \lambda_{2}<1$ be real numbers. Suppose that  $\beta:=\lbrace c_i\rbrace_{i\in I}\in \ell^{2}(I)$ is a positive sequence of real numbers. If
\begin{small}
\begin{align*}
\Vert v_{i}(C^{*}\pi_{W_{i}} C^{\prime} -C^{*} \pi_{{Z}_{i}}C^{\prime})^{\frac{1}{2}}f \Vert_{2} &\leq \lambda _{1}\Vert v_{i}(C^{*}\pi_{W_{i}} C^{\prime})^{\frac{1}{2}}f \Vert_{2}+\lambda _{2}\Vert v_{i}(C^{*} \pi_{{Z}_{i}}C^{\prime})^{\frac{1}{2}}f \Vert_{2}+\\
 & \ \ \ \ \ \ \ \ \ \ \ \ \ \ \ \ \ \ \ \ \ \ \ \ \ \ \ \ \ \ \ \ +\Vert\beta\Vert_{2}\Vert f \Vert.\\
\end{align*}
\end{small}
then, we say that $Z:=\lbrace(Z_{i},v_{i})\rbrace _{i \in I}$ is a $(\lambda _{1} ,\lambda _{2},\beta ,C ,C^{\prime})$-perturbation  of  $W=\lbrace(W_{i},v_{i})\rbrace _{i \in I}$.
\end{definition}
\begin{theorem}
Let $W:=\lbrace(W_{i},v_{i})\rbrace _{i \in I}$ be a $CC'$-controlled fusion frame for $H$ with frame bounds $A,B$, and $Z:=\lbrace(Z_{i},v_{i})\rbrace _{i \in I}$ be a $(\lambda _{1} ,\lambda _{2},\beta ,C,C^{\prime})$-perturbation  of  $W:=\lbrace(W_{i},v_{i})\rbrace _{i \in I}$. Then $Z:=\lbrace(Z_{i},v_{i})\rbrace _{i \in I}$ is a $CC'$-controlled fusion frame for  $H$ with bounds:
$$(\dfrac{(1-\lambda_{1})\sqrt{A}-\Vert\beta\Vert_{2}}{1+\lambda _{2}})^{2} \ \ , \ \ (\dfrac{(1+\lambda_{1})\sqrt{B} +\Vert\beta\Vert_{2} }{1-\lambda_{2}})^{2} $$
\end{theorem}
\begin{proof}
Let $f \in H$. We have
\begin{small}
\begin{align*}
\Vert v_{i}( C^{*} \pi_{{W}_{i}}C^{\prime})^\frac{1}{2}f \Vert_{2}&=\Vert v_{i}(C^{*}\pi_{{Z}_{i}}C^{\prime}-C^{*}\pi_{W_{i}} C^{\prime})^{\frac{1}{2}}f+v_{i}(C ^{*}\pi_{W_{i}} C^{\prime})^{\frac{1}{2}}f\Vert_{2}\\
&\leq \Vert v_{i}(C^{*} \pi_{{Z}_{i}}C^{\prime}-C ^{*}\pi_{W_{i}}C^{\prime})^{\frac{1}{2}}f \Vert_{2} +\Vert v_{i}(C^{*}\pi_{W_{i}}C^{\prime})^{\frac{1}{2}}f \Vert_{2}\\
&\leq \lambda _{1}\Vert v_{i}(C^{*}\pi_{W_{i}} C^{\prime})^{\frac{1}{2}}f \Vert_{2}+\lambda _{2}\Vert v_i(C^* \pi_{Z_{i}}C^{\prime})^{\frac{1}{2}}f \Vert_{2}+ \\
&\ \ \ \ \ \ \ \ \ \ \ +\Vert\beta\Vert_{2}\Vert f \Vert+
\Vert v_{i}(C^{*}\pi_{W_{i}} C^{\prime})^{\frac{1}{2}}f \Vert_{2}.
\end{align*}
\end{small}
Hence,
\begin{align*}
(1-\lambda_{2})\Vert (v_{i} (C^{*} \pi_{{Z}_{i}}C^{\prime})^{\frac{1}{2}}f \Vert_{2}\leq(1+\lambda_{1})\Vert v_{i}(C ^{*}\pi_{W_{i}} C^{\prime})^{\frac{1}{2}}f \Vert_{2}+\Vert\beta\Vert_{2}\Vert f \Vert.
\end{align*}
Since $W$ is a $CC'$-controlled fusion frame with bounds $A$ and $B$, then
\begin{align*}
\Vert v_{i}(C^{*}\pi_{W_{i}}C^{\prime})^{\frac{1}{2}}f \Vert^{2}
=\sum _{i \in I} v_{i}^{2}\langle \pi_{W_{i}}  C^{\prime}f, \pi _{W_{i}} C f\rangle
\leq B \Vert f \Vert^{2}.
\end{align*}
So,
\begin{small}
\begin{align*}
\Vert v_{i} (C^{*} \pi_{{Z}_{i}}C^{\prime})^{\frac{1}{2}}f \Vert_{2}&\leq\dfrac{(1+\lambda_{1})\Vert v_{i}(C^{*}\pi_{W_{i}} C^{\prime})^{\frac{1}{2}}f \Vert_{2}+\Vert\beta\Vert_{2}\Vert f \Vert }{1-\lambda_{2}}\\
&\leq(\dfrac{(1+\lambda_{1})\sqrt{B} +\Vert\beta\Vert_{2} }{1-\lambda_{2}}\Vert f \Vert).
\end{align*}
Thus
\begin{align*}
\sum _{i \in I} v_{i}^{2}\langle \pi_{Z_{i}}  C^{\prime}f, \pi _{Z_{i}} C f\rangle=\Vert v_{i} (C^{*} \pi_{{Z}_{i}}C^{\prime})^{\frac{1}{2}}f \Vert^{2}_{2}&\leq(\dfrac{(1+\lambda_{1})\sqrt{B} +\Vert\beta\Vert_{2} }{1-\lambda_{2}}\Vert f \Vert)^{2}.
\end{align*}
\end{small}
Now, for the lower bound, we have
\begin{small}
\begin{align*}
\Vert v_{i} (C^{*} \pi_{Z_{i}}C^{\prime})^{\frac{1}{2}}f \Vert_{2}&=\Vert v_{i}(C^{*}\pi_{W_{i}} C^{\prime})^{\frac{1}{2}}f -v_{i}(C^{*}\pi_{W_{i}} C^{\prime}
- C ^{*}\pi_{{Z}_{i}}C^{\prime})^{\frac{1}{2}}f\Vert_{2}\\
&\geq \Vert v_{i}(C^{*}\pi_{W_{i}}C^{\prime})^{\frac{1}{2}}f \Vert_{2}-\Vert v_{i}(C ^{*}\pi_{Z_{i}} C^{\prime}- C^{*} \pi_{{Z}_{i}}C^{\prime})^{\frac{1}{2}}f\Vert_{2}\\ &\geq \Vert v_{i}(C ^{*}\pi_{W_{i}} C^{\prime})^{\frac{1}{2}}f \Vert_{2}-\lambda _{1}\Vert v_{i}(C^{*}\pi_{W_{i}} C^{\prime})^{\frac{1}{2}}f \Vert_{2}\\
&\ \ \ \ \ -\lambda _{2}\Vert v_{i}( C^{*} \pi_{{Z}_{i}}C^{\prime})^{\frac{1}{2}}f \Vert_{2} -\Vert\beta\Vert_{2}\Vert f \Vert.
\end{align*}
\end{small}
Therefore,
\begin{align*}
(1+\lambda_{2})\Vert v_{i} (C^{*} \pi_{{Z}_{i}}C^{\prime})^{\frac{1}{2}}f \Vert_{2}\geq(1-\lambda_{1})\Vert v_{i}(C^{*}\pi_{W_{i}} C^{\prime})^{\frac{1}{2}}f \Vert_{2}-\Vert\beta\Vert_{2}\Vert f \Vert,
\end{align*}
or
$$\Vert v_{i} (C^{*} \pi_{{Z}_{i}}C^{\prime})^{\frac{1}{2}}f \Vert_{2}\geq\dfrac{(1-\lambda_{1})\Vert v_{i}(C^{*}\pi_{W_{i}} C^{\prime})f \Vert_{2}-\Vert\beta\Vert_{2}\Vert f \Vert}{1+\lambda _{2}}. $$
Thus, we get
\begin{align*}
\Vert v_{i}(C^{*}\pi_{W_{i}}C^{\prime})^{\frac{1}{2}}f \Vert^{2}=\sum _{i \in I} v_{i}^{2}\langle \pi_{W_{i}}  C^{\prime}f, \pi _{W_{i}} C f \rangle\geq  A\Vert f \Vert^{2}.
\end{align*}
So,
\begin{align*}
\Vert v_{i} (C^{*} \pi_{{Z}_{i}}C^{\prime})^{\frac{1}{2}}f \Vert_{2}\geq(\dfrac{(1-\lambda_{1})\sqrt{A}-\Vert\beta\Vert_{2}}{1+\lambda _{2}}\Vert f \Vert).
\end{align*}
Thus,
\begin{align*}
\sum _{i \in I} v_{i}^{2}\langle \pi_{Z_{i}}  C^{\prime}f, \pi _{Z_{i}} C f\rangle
&=\Vert v_{i} (C^{*} \pi_{{Z}_{i}}C^{\prime})^{\frac{1}{2}}f \Vert^{2}_{2}\\
&\geq(\dfrac{(1-\lambda_{1})\sqrt{A}-\Vert\beta\Vert_{2}}{1+\lambda _{2}}\Vert f \Vert)^{2}
\end{align*}
and the proof is completed.
\end{proof}

\end{document}